\documentclass{amsart}
\usepackage[a4paper,outer=3.4cm,inner=3.4cm, top=3.2cm, bottom =3.2cm ]{geometry}
\usepackage[english]{babel}
\usepackage[utf8]{inputenc}
\usepackage{amsmath,amsthm,amssymb}
\usepackage{enumerate}
\usepackage{tikz-cd}
\usepackage{verbatim}
\usepackage[shortlabels]{enumitem}
\setlist[enumerate]{label=\arabic*.}
\usepackage{mathtools}
\usepackage{latexsym}
\usepackage{thmtools}
\usepackage{array}
\usepackage[colorlinks=true]{hyperref}
\hypersetup{colorlinks, citecolor=blue, filecolor=black, linkcolor=black, urlcolor=black}
\usepackage{stmaryrd}

\usepackage[capitalize]{cleveref}
\crefformat{equation}{\normalfont(#2#1#3)}

\numberwithin{equation}{section}
\newtheorem{Theorem}[equation]{Theorem}
\newtheorem{Lemma}[equation]{Lemma}
\newtheorem{Proposition}[equation]{Proposition}
\newtheorem{Corollary}[equation]{Corollary}

\theoremstyle{definition}
\newtheorem{Definition}[equation]{Definition}

\newtheorem{Remark}[equation]{Remark}

\newcommand*\isomarrow{\xrightarrow{\raisebox{-0.35em}{\smash{\ensuremath{\sim}}}}}  
\newcommand{\alg}{\mathrm{alg}}

\newcommand{\perf}{\mathrm{perf}}
\newcommand{\N}{\mathbb{N}}
\newcommand{\Z}{\mathbb{Z}}
\newcommand{\Q}{\mathbb{Q}}
\newcommand{\F}{\mathbb{F}}
\newcommand{\R}{\mathbb{R}}

\renewcommand{\O}{\mathcal{O}}
\newcommand{\plim}{\textstyle \varprojlim}
\newcommand{\clim}{\textstyle \varinjlim}
\newcommand{\im}{\operatorname{im}}
\newcommand{\coker}{\operatorname{coker}}

\newcommand{\Perf}{\operatorname{Perf}}
\newcommand{\Spa}{\operatorname{Spa}}
\newcommand{\id}{{\operatorname{id}}}
\newcommand{\an}{{\mathrm{an}}}

\newcommand{\et}{{\operatorname{\acute{e}t}}}
\newcommand{\proet}{{\operatorname{pro\acute{e}t}}}
\newcommand{\qproet}{{\operatorname{qpro\acute{e}t}}}
\newcommand{\etqcqs}{{\operatorname{\acute{e}t-qcqs}}}
\newcommand{\red}{\mathrm{red}}
\newcommand{\wt}{\widetilde}

\renewcommand{\lim}{\varprojlim}
\newcommand{\aeq}{\stackrel{a}{=}}\newcommand{\cH}{{\ifmmode \check{H}\else{\v{C}ech}\fi}}
\newcommand{\Char}{\operatorname{char}}
\newcommand{\np}{{\text{-}\mathrm{np}}}

\usepackage[foot]{amsaddr}
\address{University of Frankfurt,
	Robert-Mayer-Str. 6-8,
	60325 Frankfurt am Main, Germany}
\email{heuer@math.uni-frankfurt.de}
\makeatletter
\@namedef{subjclassname@2020}{2020 MSC}
\@namedef{keywordsname}{keywords}
\makeatother
\begin{document}
	\title{The Primitive Comparison Theorem in characteristic $p$}
	\author{Ben Heuer}
	\subjclass[2020]{14G22, 14F30, 14G45.}
	\maketitle
\begin{abstract}
	We prove an analogue of Scholze's Primitive Comparison Theorem for proper rigid spaces over an algebraically closed non-archimedean field $K$ of characteristic $p$. This implies a v-topological version of the Primitive Comparison Theorem for proper finite type morphisms $f:X\to Y$ of analytic adic spaces over $\Z_p$. We deduce new cases of the Proper Base Change Theorem for \mbox{$p$-torsion} coefficients and the Künneth formula in this setting.
\end{abstract}
\section{Introduction}
Let $C$ be an algebraically closed complete field extension of $\Q_p$.
The so-called Primitive Comparison Theorem, due to Scholze \cite[Theorem~1.3]{Scholze_p-adicHodgeForRigid}\cite[Theorem 3.17]{ScholzeSurvey}, following earlier work of Faltings \cite{Faltingsalmostetale}, is a fundamental result in $p$-adic Hodge theory relating \'etale and coherent cohomology: It says that for any proper rigid space $X$ over $C$, the natural map
\[
	H^n_{\et}(X,\F_p)\otimes  \O_C/p\to H^n_{\et}(X,\O^+/p)
\]
is an almost isomorphism
for any $n\geq 0$.
The Primitive Comparison Theorem has seen plenty of applications in recent years, evidencing its importance: Originally introduced by Scholze in the proof that $H^n_\et(X,\F_p)$ is finite \cite[Theorem~5.1]{Scholze_p-adicHodgeForRigid}, it has since been used as a crucial ingredient to construct the Hodge--Tate spectral sequence \cite[Theorem~3.20]{ScholzeSurvey}, to prove \mbox{$p$-adic} Poincar\'e duality \cite[Theorem~1.1.4]{Zavyalov_acoh}\cite[Theorem 1.1.1]{mann2022padic}, various results on the arithmetic of Shimura varieties \cite[Theorem~IV.2.1]{torsion}\cite[Theorem~4.4.2]{PanLocallyAnalytic}, local constancy of  higher direct images  in \'etale cohomology \cite[Theorem~10.5.1]{ScholzeBerkeleyLectureNotes}, and Proper Base Change for $p$-torsion coefficients \cite[Theorem~3.15]{BhattHansen_Zar_constructible},  to name just a few applications.

\subsection{The Primitive Comparison Theorem over $\F_p(\!(t)\!)$}

The first goal of this article is to prove a very close analogue of the Primitive Comparison Theorem in characteristic $p$:
\begin{Theorem}\label{t:PCT-char-p-intro}
	Let $K$ be an algebraically closed non-archimedean field over $\F_p(\!(t)\!)$. Let $X$ be a  proper rigid space over $K$. Then for any $n\geq 0$, the natural map
	\[H^n_{\et}(X,\F_p)\otimes_{\F_p}\O_K/t\isomarrow H^n_{\et}(X, \O^+/t).\]
	is an almost isomorphism which is compatible with Frobenius actions.
\end{Theorem}

At first glance, this may seem surprising from the perspective that in general, the dimension of the \'etale cohomology $\dim_{\F_p}H^n_{\et}(X,\F_p)$ can be strictly smaller than that of the coherent cohomology $\dim_{K}H^n_{\et}(X,\O)$. However, we will argue that \Cref{t:PCT-char-p-intro} should be regarded as a result about the perfection $X^\perf$. In fact, we deduce it from the following stronger result:

\begin{Theorem}\label{t:thm-2-intro}
	 Let $\mathbb L$ be any $\F_p$-local system on $X_\et$. Then for any $n\geq 0$, the natural map
	\[ H^n_{\et}(X,\mathbb L)\otimes_{\F_p}\O_K\to H^n_\et(X^\perf,\mathbb L\otimes_{\F_p}\O^+)\]
	is an almost isomorphism, compatible with Frobenius actions. Second, the sequence
	\[0\to H^n_\an(X,\mathbb L\otimes\O)^{\phi\np}\to H^n_\an(X,\mathbb L\otimes\O)\to H^n_\an(X^\perf,\mathbb L\otimes\O)\to 0\]
	is short exact,
	where $-^{\phi\np}$ denotes the subspace on which the Frobenius $\phi$ is nilpotent.
\end{Theorem}
In particular, this yields a canonical and functorial decomposition
\[ H^n_\an(X,\mathbb L\otimes \O)= H^n_\an(X,\mathbb L\otimes \O)^{\phi\np}\oplus H^n_{\et}(X,\mathbb L)\otimes_{\F_p}K.\]

As an example application, \Cref{t:thm-2-intro} implies the analogue of \cite[Theorem~1.1]{Scholze_p-adicHodgeForRigid} in characteristic $p$, namely that $H^n_{\et}(X,\mathbb L)$ is a finite group for any $n\in \N$. Indeed, we have
\[\dim_{\F_p}H^n_{\et}(X,\mathbb L)\leq \dim_KH^n_\an(X,\mathbb L\otimes \O)\]
where $\mathbb L\otimes \O$ is a vector bundle, and equality holds if and only if $\phi$ is bijective on $H^n_\an(X,\O)$.
\subsection{A v-topological relative version over $\Z_p$}
Scholze deduces from the Primitive Comparison Theorem a relative version for any smooth proper morphism $f:X\to S$ of locally Noetherian adic spaces over $\Q_p$: For any $\F_p$-local system $\mathbb L$ on $X$, there is an almost isomorphism
\[R^nf_{\et\ast}\mathbb L\otimes \O^+_S/p\to R^nf_{\et\ast}(\mathbb L\otimes\O^+_X/p)\] of sheaves on $S_\et$ (we refer to \S\ref{s:almost-mathematics} for the definition of almost isomorphisms in this setting).
As a consequence of \Cref{t:PCT-char-p-intro},  this works more generally when $S$ is an analytic adic space over $\Z_p$. In fact, we also prove a more general variant of this comparison that works for the much larger v-site $S_v$ of $S$, as well as for other topologies. This is implicitly a comparison of cohomology for the  base-change $X\times_SS'\to S'$ to any perfectoid space $S'\to S$:

\begin{Theorem}[v-topological Primitive Comparison]\label{intro:PCT-perfectoid}
	Let $S$ be any analytic adic space over $\Z_p$ with a topologically nilpotent unit $\varpi\in \O(S)$. Let $f:X\to S$ be a proper morphism of finite type of adic spaces. Let $\mathbb L$ be an $\F_p$-local system on $X$. Then for any $n\geq 0$, the natural morphism
	\[ (R^nf_{v\ast}\mathbb L)\otimes_{\F_p} \O^+_S/\varpi\to R^nf_{v\ast}(\mathbb L\otimes_{\F_p} \O^+_X/\varpi)\]
is an almost isomorphism of sheaves on $S_v$.
\end{Theorem}
Here $S$ could be an adic space over $\Q_p$ or $\F_p(\!(t)\!)$, but we also allow e.g.\ ``pseudo-rigid spaces'' like $S=\Spa(\Z_p\llbracket T\rrbracket\langle \frac{T}{p}\rangle[\frac{1}{T}])$ which have geometric fibres in characteristic $0$ as well as $p$. Such spaces are of interest e.g.\ to the geometry of eigenvarieties in the context of Coleman's Spectral Halo \cite{AIP}. With an eye towards applications to $p$-adic families of Galois representations, this is one motivation to study ``$t$-adic Hodge theory'' over $\F_p(\!(t)\!)$.

\subsection{Proper Base Change}
As an application  of the Primitive Comparison Theorem  over $\F_p(\!(t)\!)$, we prove new cases of the Proper Base Change Theorem in the context of \'etale cohomology of adic spaces: Huber showed such a result in the case of prime-to-$p$-torsion coefficients \cite[Theorem~4.4.1]{huber2013etale}.
The third goal of this article is to prove the missing case of $p$-torsion sheaves (at least for local systems), following a strategy of Bhatt--Hansen \cite[\S3]{BhattHansen_Zar_constructible} and Mann \cite{mann2022padic} over $\Q_p$.
 Together with Huber's result, this combines to show:

\begin{Theorem}\label{t:proper-BC-intro}
	Let $f:X\to S$ be a proper morphism of finite type of analytic adic spaces over $\Z_p$. Let $\mathbb L$ be an \'etale-locally constant torsion abelian sheaf on $X_\et$. Let $g:S'\to S$ be any morphism of adic spaces and form the base-change diagram in locally spatial diamonds:
		\[
	\begin{tikzcd}
		X' \arrow[d, "g'"] \arrow[r, "f'"] & S' \arrow[d, "g"] \\
		X \arrow[r, "f"]                           & S.              
	\end{tikzcd}\]
	Then the base-change map
	\[
	g^{\ast}Rf_{\et\ast}\mathbb L\to Rf'_{\et\ast}g'^{\ast}\mathbb L\]
	is an isomorphism.
\end{Theorem}

As an application, this implies a K\"unneth formula for rigid spaces (\Cref{c:Kunneth}), extending the case of prime-to-$p$-torsion sheaves due to Berkovich \cite[Corollary~7.7.3]{Berkovich_EtaleCohom}.

Our concrete motivation for \Cref{intro:PCT-perfectoid} and \Cref{t:proper-BC-intro} are applications to the study of moduli spaces in $p$-adic non-abelian Hodge theory in \cite{HX}: Roughly speaking, by the idea of ``abelianisation'', it is possible in this context to understand Hodge theory of \textit{non-abelian} groups in terms of results in the \textit{relative} Hodge theory of abelian groups, such as  \Cref{intro:PCT-perfectoid}. 

\subsection*{Acknowledgments}
We thank Lucas Mann for answering all our questions about his base-change for $\O^+/p$-modules in \cite{mann2022padic}. Sections \S4 and \S5 of this article are motivated by applications to \cite{HX}, and we thank Daxin Xu for related discussions. We thank Yicheng Zhou for pointing out a mistake in an earlier version. We thank the referee for many helpful comments. Furthermore, we wish to thank Alex Ivanov, Peter Scholze and Mingjia Zhang for  helpful conversations.  		 
This project was funded by the Deutsche Forschungsgemeinschaft (DFG, German Research Foundation) -- Project-ID 444845124 -- TRR 326.
\section{Setup, Notation, Recollections}
\subsection{Rigid spaces}
Throughout, we work with adic spaces in the sense of Huber \cite{Huber-generalization}. By a non-archimedean field, we mean a pair $(K,K^+)$ consisting of a field $K$ equipped with an equivalence class of non-trivial valuations on $K$ with respect to which $K$ is complete, together with an open and integrally closed subring $K^+\subseteq K$. We can always take $K^+$ to be the ring of power-bounded elements $\O_K=K^\circ$, but we will also need to consider general $K^+$. We often omit $K^+$ from notation when it is clear from context.

By a rigid space over $K$ we mean by definition an adic space $X\to \Spa(K,K^+)$ that is locally of topologically finite type over $(K,K^+)$. Accordingly, we say that a morphism of rigid spaces is proper when it is proper in the sense of Huber, \cite[Definition 1.3.2]{huber2013etale}. When $K^+=K^\circ$, then a rigid space $X$ over $(K,K^\circ)$ is proper in this sense if and only if it is the adification of a classical rigid space  in the sense of Tate that is proper in the sense of Kiehl.

More generally, we will mainly work with analytic adic spaces $X$ over $\Z_p$. One criterion that guarantees the sheafiness in this setting is the notion of sousperfectoid spaces of \cite{HK}.

\subsection{Diamonds}
To any analytic adic space over $\Z_p$, Scholze associates in \cite[\S15]{Sch18} a locally spatial diamond $X^\diamondsuit$, which is a v-sheaf on the site of perfectoid spaces over $\F_p$. It is given by sending $Y$ to the set of isomorphism classes of tuples $((Y^\sharp,i),h:Y^\sharp\to S)$ where $Y^\sharp$ is an untilt of $Y$, $i:Y^{\sharp\flat}\isomarrow Y$ is a fixed isomorphism, and $h$ is a morphism of adic spaces. 

We often drop the $-^\diamondsuit$ from notation and denote by $X_\qproet$ and $X_v$ the quasi-pro-\'etale, respectively the v-site of $X^\diamondsuit$ in the sense of \cite[\S14]{Scholze_p-adicHodgeForRigid}. We also have the \'etale site of $X^\diamondsuit$, which is canonically isomorphic to $X_\et$ by \cite[Lemma~15.6]{Sch18}. All of these sites are ringed sites equipped with natural structure sheaves $\O_{X_\et}$, $\O_{X_\qproet}$ and $\O_{X_v}$. For example, the latter is given by sending $((Y^\sharp,i),h)$ to $\O(Y^\sharp)$. We often drop the index from notation.

\subsection{Almost mathematics}\label{s:almost-mathematics} We use almost mathematics in the sense of Gabber--Ramero \cite{GabberRamero}, in the following sheaf version from \cite[Corollary~5.12]{Scholze_p-adicHodgeForRigid}: If $Y$ is an analytic adic space, an $\O^+$-module $N$ on $Y_\et$ is defined to be almost zero if it is annihilated by the subsheaf of ideals $I\subseteq \O^+$ given on $U\in Y_\et$ by those elements of $\O^+(U)$ which have valuation $<1$ at every point of $U$. As usual, we then say that a morphism of $\O^+$-modules is an almost isomorphism if kernel and cokernel are almost zero.  When $Y$ lives over a perfectoid field $K$, this agrees with the usual definition in terms of the ideal of topologically nilpotent elements of $K^+$.  
For any diamond $Y$, we make the analogous definitions on the v-site $Y_v$.

\subsection{Cohomological comparisons}
We will frequently use the following result of Scholze:
\begin{Proposition}[{\cite[Proposition~14.7, Proposition~14.8]{Sch18}}]\label{p:14.7-14.8}
	Let $X$ be any locally spatial diamond. Let $F$ be a sheaf of abelian groups on $X_\et$ and denote by $\nu:X_\qproet\to X_\et$ and $\mu:X_v\to X_\qproet$ the natural morphisms of sites. Then for any $n\in \N$, the natural maps
	\[ H^n_\et(X,F)\isomarrow H^n_\qproet(X,\nu^\ast F)\isomarrow H^n_v(X,\mu^\ast\nu^\ast F)\]
	are isomorphisms.
	Moreover, $\nu^\ast$ and $\mu^\ast$ are fully faithful.
\end{Proposition}
For this reason, it is harmless to drop the pullbacks to simplify notation, and we will freely regard any sheaf $F$ on $X_\et$ as a v-sheaf without indicating this in the notation.

 We are particularly interested in the case that $X$ is an adic space over $\Spa(\F_p)$. In this case, the diamond $X^\diamondsuit$ is represented by a perfectoid space, namely the perfection $X^\perf$ of $X$ in the sense of \cite[Definition III.2.18]{torsion}. In particular, we need to be more careful than with \'etale sheaves when we are dealing with the natural structure sheaves: For example, the natural morphism $\O_{X_\et}\to \nu_{\ast}\O_{X_\qproet}$ is an isomorphism when $X$ is perfectoid, but not otherwise if $X$ lives over $\F_p$. On the other hand, we do have the following result for the integral subsheaves $\O^+\subseteq \O$, which we will use freely throughout the following (we note, however, that we will only need the ``almost'' version of this result, which is easier to see):
 \begin{Proposition}[{\cite[Corollary~2.15]{heuer-G-torsors-perfectoid-spaces}}]\label{p:c2.15}
 Let $X$ be any sousperfectoid adic space over $\Z_p$ that has a topologically nilpotent unit $\varpi\in \O(X)^\times$.  Then for any $n\in \N$, the natural maps 
 \[H^n_\et(X,\O^+/\varpi)\isomarrow H^n_\qproet(X,\O^+/\varpi)\isomarrow H^n_v(X,\O^+/\varpi).\]
 are isomorphisms (not just almost isomorphisms). This also holds when $X$ is a rigid space.
\end{Proposition}
\begin{proof}
	The comparison of \'etale and v-cohomology is \cite[Corollary~2.15]{heuer-G-torsors-perfectoid-spaces}. The comparison to quasi-pro-\'etale cohomology follows from this: We clearly have \[H^0_\et(X,\O^+/\varpi)\subseteq H^0_\qproet(X,\O^+/\varpi)\subseteq H^0_v(X,\O^+/\varpi)\] and the composition is an equality, so the inclusions are equalities. For $n\geq 1$, the presheaf $U\mapsto H^n_v(U,\O^+/\varpi)$ on $\Perf_K$ vanishes after \'etale sheafification, hence also after quasi-pro-\'etale sheafification. This implies the case of $n\geq 1$.
\end{proof}
\section{A Primitive Comparison Theorem in characteristic $p$}
The aim of this section is to prove the absolute version of the Primitive Comparison Theorem in characteristic $p>0$: Let $(K,K^+)$ be an algebraically closed non-archimedean field of characteristic $p$ and let $\varpi$ be a pseudo-uniformiser. Let $X$ be a proper rigid space over $(K,K^+)$. There is a natural homomorphism of $K^+/\varpi$-modules
\[ H^n_\et(X,\F_p)\otimes K^+/\varpi\to H^n_\et(X,\O^+/\varpi),\]
exactly like in the Primitive Comparison Theorem in characteristic $0$. The goal of this section is to prove that, to our surprise, this turns out to be an almost isomorphism.

\begin{Definition}\label{d:phi-np}
	Let $R$ be an $\F_p$-algebra and denote by $F$ the Frobenius. Let $M$ be any $R$-module equipped with an $F$-semilinear map $\phi:M\to M$. We denote the $\phi$-nilpotent subspace of $M$ by
	\[M^{\phi\np}:=\{m\in M\mid \phi^d(m)=0 \text{ for }d\gg 0\}.\]
	By semi-linearity, this is clearly an $R$-submodule of $M$.
	The goal of this section is to show:
\end{Definition}
\begin{Theorem}\label{t:PC-in-char-p}
	Let $X$ be any  proper rigid space over $(K,K^+)$. Let $\mathbb L$ be an $\F_p$-local system on $X$. Then for any $n\geq 0$:
	\begin{enumerate}
		\item The following natural map is an almost isomorphism compatible with Frobenius:
		\[H^n_{\et}(X,\mathbb L)\otimes_{\F_p}K^+/\varpi\aeq H^n_{\et}(X,\mathbb L\otimes \O^+/\varpi).\]
		\item The following natural map is an almost isomorphism compatible with Frobenius:
		\[ H^n_{\et}(X,\mathbb L)\otimes_{\F_p}K^+\aeq H^n_{\et}(X^{\perf},\mathbb L\otimes \O^+).\]
		\item Let $\phi=\id\otimes F$ on $\mathbb L\otimes \O$. Then there is a short exact sequence of $K$-vector spaces
		\[0\to H^n(X,\mathbb L\otimes \O)^{\phi\np}\to H^n(X,\mathbb L\otimes \O)\rightarrow H^n(X^{\perf},\mathbb L\otimes \O)\to 0.\]
	\end{enumerate}
\end{Theorem}

The key idea for the proof is that the right place to look for a Primitive Comparison Theorem in characteristic $p$ is the perfection $X^\perf$. Indeed, part 2 can be interpreted as a ``Primitive Comparison Theorem for the perfection $X^\perf$''. The analogous statement of part 2 clearly becomes false if we replace $X^\perf$ by $X$: For example, for a supersingular elliptic curve, we have $H^1_\et(X,\F_p)=0$ but $\dim_K H^1_\et(X,\O)\neq 0$. In particular, the natural map
\[H^n_\et(X,\O^+)\to  H^n_\et(X^\perf,\O^+)\]
is in general not injective.
However, by \Cref{p:c2.15}, we do have an isomorphism
\[H^n_\et(X,\O^+/\varpi)= H^n_v(X,\O^+/\varpi)= H^n_\et(X^\perf,\O^+/\varpi).\]
Indeed, this can be seen directly by considering the Frobenius $F:\O^+/\varpi\to \O^+/\varpi^p$ on $X_\et$. It is clearly injective. It is surjective because for any local section $f$ of $\O^+$, the equation $T^p+\varpi T+f$ defines a finite \'etale cover of $X$ over which $f$ acquires a $p$-th root $T$ in $\O^+/\varpi$.

Finally, we interpret part 3 of \Cref{t:PC-in-char-p} as saying that the comparison of \'etale and coherent cohomology on $X^\perf$ still gives a useful statement for $X$ because the coherent cohomology of $X^\perf$ can be described explicitly in terms of that of $X$. More explicitly:
\begin{Corollary}
	We have
	\[H^n_\et(X,\F_p)\otimes_{\F_p}K=H^n_\et(X,\O)/H^n_\et(X,\O)^{\phi\np}.\] In particular, \[\dim_{\F_p} H^n_\et(X,\F_p)\leq \dim_KH^n_\et(X,\O)\] and equality holds if and only if $\phi$ is bijective on $H^n_\et(X,\O)$.
\end{Corollary}
We deduce the following finiteness result which is an analogue of \cite[Theorem~1.1]{Scholze_p-adicHodgeForRigid}:

\begin{Corollary}\label{c:H^iXF-char-p-finite}
	For any $n\in \N$, the group $H^n_\et(X,\mathbb L)$ is finite. It vanishes for $n>\dim X$.
\end{Corollary}
We note that the vanishing is in degrees $> \dim X$, rather than for $> 2\dim X$ as is the case for $\F_l$-coefficients for $l\neq p$, or in the case of characteristic $0$.
\begin{proof}
By Kiehl's Proper Mapping Theorem, $H^n_\et(X,\O)$ is a finite dimensional $K$-vector space. By \cite[Proposition 2.5.8]{deJongvdPut}, it vanishes for $n>\dim X$.
\end{proof}

\begin{Corollary}
	Let $n\in \N$ and set $r(n):=\dim_{\F_p}H^n(X,\F_p)$.
	 Then 
	for $F$ any one of the sheaves  $\O$, $\O^{+a}$, $\O^{+a}/\varpi$ on $\Perf_K$, there are compatible Frobenius-equivariant isomorphisms
	\[ H^n(X^\perf, F)\cong F(K)^{r(n)}.\]
\end{Corollary}
\begin{Corollary}\label{c:RH}
	We have a canonical identification
	\[H^n_\et(X,\mathbb L)=H^n(X,\mathbb L\otimes \O)^{\phi=\id}.\]
\end{Corollary}
\begin{Remark}
	\Cref{c:RH} bears a close relation to the ``Riemann--Hilbert correspondence in characteristic $p$'':
	\begin{enumerate}
		\item It is closely related to the  Riemann--Hilbert correspondence of Mann \cite[\S3.9, Theorem 3.9.23]{mann2022padic}, which also gives this statement, at least when $X$ is smooth.
		\item 
		When $X$ is algebraic, \Cref{c:RH} is closely related to the Riemann--Hilbert correspondence of Bhatt--Lurie: In particular, via the usual comparison theorems for \'etale cohomology,  \cite[Theorem 10.5.5]{BhattLurie}  also implies 	\Cref{c:RH} in the algebraic case.
	\end{enumerate}
\end{Remark}

Finally, we obtain the following general vanishing result for ``supersingular'' varieties:

\begin{Corollary}
	Let $n\in \N$. Let $X$ be a proper rigid space over $K$ such that the Frobenius on $H^n(X,\O)$ is nilpotent. Let $F$ be any one of the following sheaves on $X^\perf$: $\O$, $\O^{+a}$,  $\O^{+a}/\varpi$, $\Z_p$, $\Q_p$, $\Z/p^d$ for some $d\in \N$, or $\O^{\sharp}$, $\O^{+a,\sharp}$,  with respect to a fixed untilt of $K$. Then \[H^n(X^\perf,F)=0.\]
\end{Corollary}
\begin{proof}
	The case of $\O$ follows from \Cref{t:PC-in-char-p}.3, the case of $\Z/p$ from \Cref{t:PC-in-char-p}.2. All other cases follow from these by long exact sequences and limits.
\end{proof}
\Cref{c:H^iXF-char-p-finite} combines with previous results of Huber and Scholze to show the following:

\begin{Theorem}\label{t:finiteness}
	Let $X\to \Spa(L,L^+)$ be a proper rigid space over \text{any} non-archimedean field $(L,L^+)$.
	Let $F$ be any \'etale-locally constant torsion abelian sheaf of finite type on $X_\et$. Then for any $n\in \N$, the group $H^n_\et(X,\mathbb L)$ is finite and vanishes for $n>2\dim X$.
\end{Theorem}
\begin{proof}
	We can decompose $F$ as the direct sum of its $\ell$-primary parts for each prime $\ell$ to reduce to the case that $F$ is $\ell$-power torsion. By considering $\ell$-torsion subsheaves, we see that $F$ is a successive extension of $\mathbb F_\ell$-local systems, hence we can reduce to this case.   Let $p$ be the characteristic of the residue field of $L$. We now distinguish three cases:
	
	In case that $\ell\neq p$, the result is due to Huber: By \cite[Proposition 8.2.3(ii)]{huber2013etale}, we can without loss of generality assume that $L^+=L^\circ$. Then the result is \cite[Corollary~5.4]{HuberFiniteness}.
	The vanishing for $n>2\dim X$ is also due to Huber, \cite[Corollary 2.8.3]{huber2013etale}.
	
	In case that  $\ell=p$ and $\Char(K)=0$, the result is due to Scholze, \cite[Theorem~1.1]{Scholze_p-adicHodgeForRigid} in the context of the Primitive Comparison Theorem.
	
	In case that  $\ell=p$ and $\Char(K)=p$, the statement is \Cref{c:H^iXF-char-p-finite}.
\end{proof}

\subsection{$\O^+$-cohomology is finitely generated}
In this subsection, we can work with a general non-archimedean field $K$ over $\Z_p$ with non-discrete value group $\Gamma\subseteq \R$. Let $\O_K=K^\circ$ be the ring of integers. Let $\varpi\in \O_K$ be a pseudo-uniformiser.

For the proof of Theorem~\ref{t:PC-in-char-p}, we will draw on the strategies of \cite[\S2 and \S5]{Scholze_p-adicHodgeForRigid} and the one surrounding \cite[Proposition 6.10]{perfectoid-spaces}, supplemented by some additional semi-linear algebra input. We start with the following variant of the definitions in \cite[\S2]{Scholze_p-adicHodgeForRigid}:
\begin{Definition}\label{d:a-fin-free}
	Let $M,N$ be $\O_K$-modules and let $\epsilon\in \Gamma_{>0}:=\Gamma\cap \R_{>0}$. We fix an element $\varpi^\epsilon\in K$ such that $|\varpi^{\epsilon}|=|\varpi|^{\epsilon}$.
	\begin{enumerate}
		\item We write $M\approx_{\epsilon} N$ if there are morphisms of $\O_K$-modules $f_{\epsilon}:M\to N$, $g_{\epsilon}:N\to M$ such that $f_{\epsilon}g_{\epsilon}=\varpi^{\epsilon}\id_N$, $g_{\epsilon}f_{\epsilon}=\varpi^{\epsilon}\id_M$. We write $M\approx N$ if $M\approx_{\epsilon}N$ for all $\epsilon \in\Gamma_{>0}$.
		\item We say that $M$ is almost finitely generated (resp.\ presented) if for all $\epsilon\in \Gamma_{>0}$, there is a finitely generated (resp.\ presented) $\O_K$-module $N_{\epsilon}$ such that $M\approx N_{\epsilon}$.
		\item We say that $M$ is almost finite free of rank $r$ if $M\approx \O_K^r$ for some $r\in \N$.
		\item $M$ has bounded torsion if there is $k\in \N$ such that $M[\varpi^k]=M[\varpi^l]$ for all $l>k$.
		\item We say that $M$ is bounded torsion if there is $k\in \N$ such that $M=M[\varpi^k]$.
	\end{enumerate}
\end{Definition}
\begin{Lemma}\label{l:submod-of-O^r-aff}
	Any bounded $\O_K$-submodule $M$ of $K^r$ is almost finite free.
\end{Lemma}
\begin{proof}
	After rescaling, we may assume that $M\subseteq \O_K^r$.
	By  \cite[Proposition~2.8.(ii)]{Scholze_p-adicHodgeForRigid}, it suffices to prove that $M$ is almost finitely presented.
	The case of $r=1$ is \cite[Example~2.4.(i)]{Scholze_p-adicHodgeForRigid}. The statement then follows inductively using \cite[Lemma~2.7.(ii)]{Scholze_p-adicHodgeForRigid}.
\end{proof}
\begin{Lemma}\label{l:ses-for-H^i(X,O+)}
	Let $X$ be a reduced proper rigid space over $\Spa(K,\O_K)$. Let $V$ be a finite locally free $\O^+$-module on $X_\et$. 
	Let $\mathcal U$ be an \'etale cover of $X$ by finitely many affinoid rigid spaces $X_i\to X$. Then for any $n\geq 0$, there is a short exact sequence
	\[ 0\to M\to \cH^n(\mathcal U,V)\to N\to 0\]
	where $M$ is bounded torsion and $N$ is an almost finite free $\O_K$-module.
\end{Lemma}
\begin{proof}
	Consider the \cH\ complex $C^+:= \check C^{\ast}(\mathcal U,V)$. The complex $C:= C^+[\tfrac{1}{\varpi}]$ obtained by inverting $\varpi$ is a complex of Banach $K$-vector spaces that computes the coherent cohomology $H^n(X,V[\tfrac{1}{\varpi}])$ where $V[\tfrac{1}{\varpi}]$ is a vector bundle. By Kiehl's Proper Mapping Theorem, this is a finite dimensional $K$-vector  space for all $n$.
	By \cite[\S II.5 Lemma 1]{MumfordAV}, it follows that there is a bounded complex of finite dimensional $K$-vector spaces $L$ with a quasi-isomorphism
	\[ f:L\to  C.\] 
	Due to the boundedness, we can choose a subcomplex $L^+\subseteq L$ of finite free $K^+$-modules such that $L=L^+[\tfrac{1}{\varpi}]$ and such that  $f$ has an integral model
	$f^+:L^+\to  C^+$.
	Since $f$ is a quasi-isomorphism, the mapping cone $\mathrm{Cone}(f^+)$ becomes an exact complex of Banach $K$-vector spaces after inverting $\varpi$. Note that since any space $U$ in the \v{C}ech nerve of $\mathcal U$  is a reduced affinoid rigid space, $\O^+(U)\subseteq \O(U)$ is open and bounded.	
	Arguing exactly as in \cite[Proposition~6.10.(iii)]{perfectoid-spaces}, it  follows using the Banach Open Mapping Theorem that the cohomology groups of $\mathrm{Cone}(f^+)$ are bounded torsion. Since $L^+$ is a bounded complex of finite free $K^+$-modules, it follows from exact sequences that $H^n:=H^n(C^+)$ has bounded torsion. Setting $M=H^n[\varpi^\infty]$ and $N:=H^n/M$, we see that $N$ is a bounded $\O_K$-submodule of a finite-dimensional $K$-vector space. Thus $N$ is almost finite free by Lemma~\ref{l:submod-of-O^r-aff}
\end{proof}
\subsection{Proof of  Theorem~\ref{t:PC-in-char-p}}
Returning to the setup of Theorem~\ref{t:PC-in-char-p}, we now assume that $K$ is an algebraically closed non-archimedean field of characteristic $p$. Let us denote by $F$ the Frobenius on $K$. We begin  with some preparations for the proof of \Cref{t:PC-in-char-p}.3.
\begin{Lemma}\label{l:colim_phi-M-almost-finite-free}
	\begin{enumerate}
		\item Let $V$ be a finite dimensional $K$-vector space with an $F$-semilinear map $\phi:V\to V$. Then there is a short exact sequence of $K$-vector spaces
		\[ 0\to V^{\phi\np}\to V\to \clim_{\phi}V\to 0,\]
		where $V^{\phi\np}:=\{v\in V\mid \phi^d(v)=0 \text{ for }d\gg 0\}$ was defined in Definition~\ref{d:phi-np}.
		\item Let $M$ be an almost finite free $\O_K$-module and let $\phi:M\to M$ be an $F$-semilinear map. Then there is an exact sequence
  		\[ 0\to M^{\phi\np} \to M\to \clim_{\phi}M\to L\to 0\]
		where $L$ is almost finitely presented. In particular, $\clim_{\phi}M$ is almost finite free.
	\end{enumerate}
\end{Lemma}
\begin{proof}
	Part 1 can be deduced from the theory of $\phi$-modules as developed in \cite{Kedlaya_differential-equations}\cite{BhattLurie} but it is also easy to see directly: $(\ker \phi^d)_{d\in \N}$ is an increasing chain of sub-vector spaces of $V$. Since $V$ is finite dimensional, this chain stabilises for some $d\gg 0$. Thus also $\im \phi^d$ stabilises and one checks directly that it gets identified with $\clim_{\phi}V$. The desired sequence is now simply $0\to \ker \phi^d\to V\to \im \phi^d\to 0$ for $d\gg 0$.
	This proves part 1. 
	
	Towards part 2, we first note that, in particular,  $U:=\clim_{\phi} V$ is a finite dimensional \mbox{$K$-vector} space on which $\phi$ is a bijective $F$-semilinear map. Since $K$ is algebraically closed, it follows that there is a basis of elements $v\in U$ on which $\phi(v)=v$: This is the special case of $n=1$ and $S_n=\mathrm{Spec}(K)$ of \cite[Proposition~4.1.1]{p-adicMSMF}. See also \cite[Theorem 2.4.3]{BhattLurie}.
	
	We can use this to prove the second part as follows: Left-exactness of the sequence is clear from 1. Replacing $\phi$ by $\phi^d$ and $M$ by $M/\ker \phi^d$, we may therefore without loss of generality assume that $\phi$ is injective. By part 1, it is then bijective after inverting $\varpi$, so we may find a basis of $M[\tfrac{1}{\varpi}]$ of elements $v_i\in M$ for which $\phi(v_i)=a_iv_i$ for some $a\in K$. Choose now an injection $i:\O_K^r\to M$ with $\varpi$-torsion cokernel which exists by \Cref{d:a-fin-free}, and let $M_0$ be its isomorphic image. Then we may rescale the $v_i$ so that their $\O_K$-span $N$ is contained in $M_0$, and so that $a_i\in \O_K$ for all $i$. Then $N\to M_0\to M$ is an injection whose cokernel $C$ is annihilated by $\varpi^k$ for some $k\in \N$. Thus, in the limit, we obtain a sequence
	\[ 0\to\clim_{\phi}N\to \clim_{\phi}M\to \clim_{\phi} C\to0\]
	in which the last term is $\aeq 0$: Indeed, the image of $\phi^{d}:C\to C$ is $\varpi^{k/p^d}$-torsion with respect to the module structure via $F^d:K\to K$. 
	
	We have thus reduced to the case that $M=N=\O_K^r$ is spanned by $\O_K$-basis vectors $v_i$ for which $\phi(v_i)=a_iv_i$. We can thus reduce to $r=1$ and $\phi:\O_K\to \O_K$ being given by $a\cdot F$ for some $a\in \O_K$. We then see that as an $\O_K$-module via $F^d:\O_K\to \O_K$, we have
	\[ \coker(\O_K\xrightarrow{\phi^d} \O_K)=a^{-\sum_{i=0}^{d-1} p^{-i}}\O_K/\O_K.\]
	It follows that 
	\[L=\coker(\O_K\to \clim_{\phi}\O_K)\approx a^{-\frac{p}{p-1}}\O_K/\O_K.\]
	This is clearly almost finitely presented.
	Finally, it follows that $\clim_{\phi}M$ is an extension of almost finitely presented modules, thus is itself almost finitely presented, thus almost finite free by \cite[Proposition 2.8.(ii)]{Scholze_p-adicHodgeForRigid} since it is torsionfree.
\end{proof}

\begin{Proposition}\label{p:H^i(X^perf,O^+)}
	Let $X$ be any proper rigid space over $(K,\O_K)$. Let $\mathbb  L$ be an $\F_p$-local system on $X_\et$. Then for all $n\geq 0$, we have:
	\begin{enumerate}
		\item There is $r\in \N$ and  an $\O_K$-linear almost isomorphism,
		compatible with Frobenius,
		\[ H^n_\et(X^{\perf},\mathbb L\otimes \O^+)\aeq \O_K^r.\]
		\item For  any \'etale cover $\mathcal U$ of $X$ by finitely many affinoids trivialising $\mathbb L$, we have \[\clim_{\phi}\cH^n(\mathcal U,\mathbb L\otimes \O^+)\approx H^n_\et(X^{\perf},\mathbb L\otimes \O^+).\]
	\end{enumerate}
\end{Proposition}
\begin{proof}We first observe that we may without loss of generality assume that $X$ is reduced: Indeed, let $X^{\red}\to X$ be the reduction, then we have $X^{\red}_{\et}=X_\et$. With respect to this identification, there is $k\in\N$ such that the Frobenius $F^k:\O_{X}^+\to \O_{X}^+$ factors through $\O_{X^\red}^+$. It follows that $X^{\perf}=X^{\red,\perf}$, and $\varinjlim_\phi\cH^n(\mathcal U,\mathbb L\otimes \O^+_X)=\varinjlim_\phi\cH^n(\mathcal U,\mathbb L\otimes \O_{X^\red}^{+})$.
	
	Set $E:=\mathbb L\otimes \O^+$, this is a finite locally free $\O^+$-module on $X_\et$. 
	For any \'etale map $Y\to X$, we have $Y\times_XX^\perf=Y^\perf$. Let $\mathcal U^\perf$ be the cover of $X^\perf$ obtained by perfecting $\mathcal U$.
	By comparing \cH\ cohomology, it follows from almost acyclicity of $\O^+$ on affinoid perfectoid spaces that 
	\[R\Gamma(X^{\perf},E)\stackrel{a}{\simeq} \check{ C}^\ast(\mathcal U^\perf,E)\aeq  (\clim_\phi\check{ C}^\ast(\mathcal U,E))^\wedge\]
	where $-^\wedge$ denotes the term-wise (i.e.\ underived) $\varpi$-adic completion.
	Since $X$ is reduced, we may apply Lemma~\ref{l:ses-for-H^i(X,O+)}: For any $n$, we have a short exact sequence
	\[ 0\to M\to \cH^n(\mathcal U,E)\to N\to 0\]
	where $N$ is almost finite free and $M$ is bounded torsion. Upon applying $\clim_\phi$, the first term thus becomes $\varinjlim_\phi M\aeq 0$. By Lemma~\ref{l:colim_phi-M-almost-finite-free}, we deduce that
	\[ W:=H^n(\clim_{\phi}\check{ C}^\ast(\mathcal U,E))=\clim_{\phi}\cH^n(\mathcal U,E)\approx \clim_{\phi}N\]
	is an almost finite free $\O_K$-module of rank $r$ (we may commute $\cH$ and $\clim$ since the cover $\mathcal U$ is finite).
	In particular, the cohomology of $\clim_{\phi}\check{ C}^\ast(\mathcal U,E)$  is almost torsionfree in all degrees.
	
	For any $k\in \N$ we now set
	\[M_k:=H^n_\et(X^\perf,E/\varpi^k).\]
	Then $M_k$ is computed up to almost isomorphism by the complex \[\check{ C}^\ast(\mathcal U^\perf,E)/\varpi^k=(\clim_{\phi}\check{ C}^\ast(\mathcal U,E))/\varpi^k\]
	where the quotient by $\varpi^k$ is underived. Since $\clim_{\phi}\check{ C}^\ast(\mathcal U,E)$ consists of flat $\O_K^a$-algebras and the cohomology is almost finite free in all degrees, we may commute cohomology and quotient and deduce that
	\[M_k\approx W/\varpi^k.\]
	In particular, $M_k$ is almost finitely generated.  
	Moreover, the system of $\O_K$-modules $(M_k)_{k\in \N}$ has the following properties:
	\begin{itemize}
		\item The natural maps $p_k:E/\varpi^{k+1}\to E/\varpi^{k}$ and $q_k:E/\varpi^{k}\xrightarrow{\cdot \varpi} E/\varpi^{k+1}$ induce maps $p_k:M_{k+1}\to M_k$ and $q_k:M_k\to M_{k+1}$ such that $q_k\circ p_k=\varpi$ and the sequence
		\[ M_1\xrightarrow{q_{k}\circ\dots\circ q_1} M_{k+1}\xrightarrow{p_k} M_k\]
		is exact in the middle.
		\item The absolute Frobenius $F_k:\O^+/\varpi^{k}\isomarrow \O^+/\varpi^{pk}$ induces isomorphisms
		\[ F_k:M_k\otimes_{\O_K/\varpi^k}\O_K/\varpi^{pk}\isomarrow M_{pk}.\] 
	\end{itemize}
	We are thus in a position to apply \cite[Lemma~2.12]{Scholze_p-adicHodgeForRigid}, which deduces from the above input that there are compatible almost isomorphisms (not just $\approx$ relations) 
	\[ M_k^a\isomarrow (\O^a_K/\varpi^k)^{r}\]
	such that $p_k$, $q_k$ and $F_k$ are identified with the obvious maps on the right hand side. Moreover, as the proof shows, these are induced by an isomorphism $\varprojlim_k M_k^a\cong (\O_K^a)^r$. Next, since $X^\perf$ is perfectoid, we have $H^n_\et(X^\perf,E)=H^n_v(X^\perf,E)$
	and using repleteness of $X^{\perf}_{v}$ (\cite[\S3]{bhatt-scholze-proetale}), it follows that there is an almost isomorphism
	\[H^n_\et(X^\perf,\mathbb L\otimes\O^+)\aeq \plim_k H^n_\et(X^\perf,\mathbb L\otimes\O^+/\varpi^k)\aeq \plim_kM_k\aeq \O_K^r\]
	that matches up the Frobenius morphisms on both sides. Finally, since $W$ is almost finitely free, it follows that 
	\[\textstyle W\approx \varprojlim_k W/\varpi^k\approx \varprojlim_k M_k\aeq H^n_\et(X^\perf,\mathbb L\otimes\O^+).\qedhere\]
\end{proof}
As we have worked over $\O_K$ in the above, we need a version of an argument appearing in the proof of \cite[Theorem~5.1]{Scholze_p-adicHodgeForRigid} to reduce the general case of $(K,K^+)$ to this:

\begin{Lemma}\label{l:reduce-PCT-(C-Cp)-to-(C-OC)}
	Let $K$ be any non-archimedean field over $\Z_p$ with pseudo-uniformiser $\varpi$.
	Let $X\to \Spa(K,K^+)$ be a separated adic space. Let $\mathbb L$ be an $\F_p$-local system on $X$. Let $X'\to \Spa(K,\O_K)$ be the base-change to the subspace $\Spa(K,\O_K)\hookrightarrow \Spa(K,K^+)$. Then
	\begin{enumerate}
		\item $H^n_\et(X,\mathbb L)=H^n_\et(X',\mathbb L)$,
		\item $H^n_v(X,\O^+/p\otimes \mathbb L)\aeq H^n_v(X',\O^+/p\otimes \mathbb L)$.
	\end{enumerate}
\end{Lemma}
Here, to avoid the question whether any object in $X_\et$ is sheafy, we regard $X$ as a diamond.
\begin{proof}
	Part 1 is \cite[Proposition 8.2.3(ii)]{huber2013etale}]. Alternatively, working in $X_v$, it follows from the case of $\O^\flat$ by the Artin-Schreier sequence $0\to \F_p\to \O^\flat\to \O^\flat\to 0$ tensored with $\mathbb L$.
	
	For part 2, it suffices by considering \v{C}ech nerves of \'etale covers to prove the statement when $X$ is affinoid and $\mathbb L=\F_p$. We first deal with the case that $\mathrm{Char}(K)=0$:  By considering the \v{C}ech nerve of any affinoid pro-\'etale perfectoid cover $\wt X\to X$, we can reduce to the case that $X$ is affinoid perfectoid. Here the case $n=0$ follows from the explicit formula for fibre products of affinoid perfectoid spaces \cite[Proposition 6.18]{perfectoid-spaces}, and the higher cohomology vanishes by almost acyclicity. 
	
	For $\mathrm{Char}(K)=p$, since we work in the v-topology, we may replace $K$ and $X$ by their completed perfections, so again the statements follow from the perfectoid case.
\end{proof}

\begin{proof}[Proof of Theorem~\ref{t:PC-in-char-p}]
	By \Cref{l:reduce-PCT-(C-Cp)-to-(C-OC)}, we may assume that $K^+=\O_K$. In this case, Proposition~\ref{p:H^i(X^perf,O^+)} says that
	\[ H^n_\et(X^\perf,\mathbb L\otimes \O^+)\aeq \O_K^{r}\aeq K^{+r}\]
	for some $r$.
	Part 2 now follows from the Artin--Schreier exact sequence
	\[ 0\to \mathbb L\to \mathbb L\otimes\O\xrightarrow{\id \otimes (x\mapsto x^p-x)}\mathbb L\otimes\O\to 0\]
	on $X^\perf_{\proet}$ like in \cite[Theorem~5.1]{Scholze_p-adicHodgeForRigid}: 
	Since the almost isomorphism of Proposition~\ref{p:H^i(X^perf,O^+)} respects the Frobenius action, the cohomology sequence in degree $n$ becomes identified with the short exact sequence
	\[ 0\to \F_p^r\to K^{r}\xrightarrow{x\mapsto x^p-x} K^{r}\to 0\]
	for $r=\dim_{\F_p}H^n_\et(X,\mathbb L)$.
	It is thus itself short exact, proving that 
	\[ H^n_\et(X^\perf,\mathbb L)=H^n_\et(X^\perf,\mathbb L\otimes \O)^{\phi=1}\cong \F_p^r.\]
	Now the same Artin--Schreier sequence for $\O^+$ instead of $\O$ shows part 2. 
	
	For part 1, we first consider the long exact sequence of the sequence \[0\to \O^+\xrightarrow{\cdot \varpi} \O^+\to \O^+/\varpi\to 0\] on $X^\perf_\et$. Using the 5-Lemma, we deduce from part 2 that this yields an almost isomorphism
	\[ H^n_\et(X^\perf,\mathbb L)\otimes K^+/\varpi\isomarrow H^n_\et(X^\perf,\mathbb L\otimes \O^+/\varpi).\]
	We now use that $X_\et = X_\et^\perf$, hence  $ H^n_\et(X,\mathbb L)= H^n_\et(X^\perf,\mathbb L)$. On the other hand, we have $H^n_\et(X,\mathbb L\otimes \O^+/\varpi)=H^n_\et(X^\perf,\mathbb L\otimes \O^+/\varpi)$ by \Cref{p:c2.15}.
	
	For part 3, let $\mathcal U$ be an \'etale cover of $X$ by finitely many affinoids trivialising $\mathbb L$. Set $\cH^n:=\cH^n(\mathcal U,\mathbb L\otimes \O^+)$, then  \[V:=\cH^n[\tfrac{1}{\varpi}]=\cH^n(\mathcal U,\mathbb L\otimes \O)=H^n(X,\mathbb L\otimes \O).\]
	By Proposition~\ref{p:H^i(X^perf,O^+)}, \[\clim_{\phi} V= H^n_\et(X^\perf,\mathbb L\otimes \O).\] The short exact sequence of Lemma~\ref{l:colim_phi-M-almost-finite-free}.1 now gives  the desired sequence
	\[ 0\to H^n(X,\mathbb L\otimes \O)^{\phi \np}\to H^n(X,\mathbb L\otimes \O)\to H^n(X^\perf,\mathbb L\otimes \O)\to 0.\]
	This finishes the proof of \Cref{t:PC-in-char-p}.
\end{proof}

\section{A v-topological Primitive Comparison Theorem}
Building on earlier work of Faltings, Scholze's Primitive Comparison Theorem was first formulated in \cite[Theorem~5.1]{Scholze_p-adicHodgeForRigid} as saying that for a smooth proper rigid space $X$ over an algebraically closed non-archimedean field $(C,C^+)$ over $\Q_p$, one has  an almost isomorphism
\[ H^n_\et(X,\mathbb F_p)\otimes C^+/p=H^n_\et(X,\O^+_X/p).\]
Scholze deduced a relative version \cite[Corollary~5.12]{Scholze_p-adicHodgeForRigid}, and later removed the smoothness assumption. We can combine this with \Cref{t:PC-in-char-p} to obtain the following:
\begin{Theorem}\label{t:PCT-rigid}
	Let $S$ be a locally Noetherian analytic adic space over $\Z_p$ with a topologically nilpotent unit $\omega\in \O(S)^\times$. 
	Let $f:X\to S$ be a proper morphism of finite type of adic spaces. Let $\mathbb L$ be an $\F_p$-local system on $X$. Then for any $n\in \N$,
	the natural morphism
	\[(R^nf_{\et\ast}\mathbb L)\otimes_{\Z_p} \O^+_S/\varpi\to R^nf_{\et\ast}(\mathbb L\otimes_{\Z_p} \O^+_X/\varpi) \] 
	is an almost isomorphism of sheaves on $S_\et$.
\end{Theorem}
\begin{proof} 
	In the case that $S$ lives over $\Q_p$, this is due to Scholze. As it is not recorded in the literature in this generality, we sketch how to deduce the statement in this case from  {\cite[\S5]{Scholze_p-adicHodgeForRigid}}{\cite[\S3]{ScholzeSurvey}}{\cite[\S6.3]{Zavyalov_acoh}}, without claiming any originality (any mistakes are due to us):
	
	When $\mathbb L=\mathbb F_p$ and $\varpi=p$ and $f$ is smooth,
	this is proved  in \cite[Corollary~5.12]{Scholze_p-adicHodgeForRigid} by deducing it from the absolute case of $S=\Spa(C,C^+)$, which is \cite[Theorem~5.1]{Scholze_p-adicHodgeForRigid}.
	
	When $S=\Spa(C,\O_C)$ and $\mathbb L=\mathbb F_p$ and $\varpi=p$ but $f$ is not necessarily smooth, it is \cite[Theorem~3.17]{ScholzeSurvey}. The case of general $\mathbb L$ follows from this by the following strategy explained in \cite[Lemma~6.3.7]{Zavyalov_acoh}: 
	By arguing as in \cite[0A3R]{StacksProject}, we can find a finite \'etale morphism  $g:X'\to X$ of rigid spaces of degree coprime to $p$ on which $g^\ast\mathbb L$ admits a filtration whose graded pieces are isomorphic to $\F_p$.  Then $\mathbb L$ is a direct factor of $g_\ast g^{\ast}\mathbb L$ via the normalised trace map. 
	As a direct sum of maps is an almost isomorphism if and only if each factor is,  it suffices to prove the result for $g_\ast g^{\ast}\mathbb L$. By induction on the length of the filtration we can therefore further reduce to $\mathbb L=g_\ast \F_p$.
	
	By \cite[Corollary 2.6.6]{huber2013etale}, we then have $\mathbb L=Rg_\ast \mathbb F_p$, hence \[Rf_{\ast}\mathbb L=R(f\circ g)_{\ast}\mathbb F_p.\] By the smooth case of \Cref{t:PCT-rigid} applied to $g$ and $\F_p$, we also have 
	\begin{alignat*}{3}
	\mathbb L\otimes \O^+/p=Rg_\ast \mathbb F_p \otimes \O^+/p&\aeq Rg_\ast ( \O^+/p),\\
	\Rightarrow \quad Rf_{\ast}(\mathbb L\otimes \O^+/p)&\aeq R(f\circ g)_{\ast}\O^+/p.
	\end{alignat*}
	 This shows the Theorem for $S=\Spa(C,\O_C)$, $\varpi=p$.
	
	Next, we consider general $0\neq \varpi\in C$ for $S=\Spa(C,\O_C)$: The Frobenius morphism defines an isomorphism \[F^k:\O^+/p^{1/p^k}\to \O^+/p\]
	 which shows the case of $\varpi=p^{1/p^k}$. By the 5-Lemma, we deduce the case of $\varpi=p^{a/p^k}$ for any $a\in \N$. Since $\Z[\tfrac{1}{p}]_{\geq 0}$ is dense in the value group of $C$, this is enough to get an almost isomorphism for general $\varpi$.
	
	 By \Cref{l:reduce-PCT-(C-Cp)-to-(C-OC)} and \Cref{p:c2.15}, the case of $S=\Spa(C,\O_C)$ implies that of $\Spa(C,C^+)$.
	
	The general relative case can now be deduced from this and \Cref{t:PC-in-char-p} following the proof of \cite[Corollary~5.12]{Scholze_p-adicHodgeForRigid}, \cite[Lemma 6.3.7]{Zavyalov_acoh}. We postpone the details to the proof of \Cref{t:PCT-perfectoid} below, where we will give a more general version of this kind of argument.
\end{proof}
\begin{Remark}
	Mann has recently given in \cite[Corollary~3.9.24]{mann2022padic}  a generalisation of the Primitive Comparison Theorem from morphisms of adic spaces over $\Q_p$ to certain ``bdcs'' maps of $p$-bounded locally spatial diamonds, phrased in terms of his 6-functor formalism. While this is very general, it is not clear whether the assumption that  $f^!$ preserves all small colimits applies in the characteristic $p$ setup over $\F_p(\!(t)\!)$, which is the main focus of this article.
\end{Remark}

The main goal of this section is to prove the following improvement of \Cref{t:PCT-rigid} which allows more general topologies:
\begin{Theorem}\label{t:PCT-perfectoid}
	Let $S$ be an analytic adic space over $\Z_p$ with a unit $\varpi\in \O(S)^\times$ that is topologically nilpotent locally on $S$.
	Let $f:X\to S$ be a proper morphism of finite type of adic spaces. Let $\mathbb L$ be an $\F_p$-local system on $X$.   We endow $S$ with one of the following topologies:
	\begin{enumerate}
		\item $\tau=v$, or
		\item $\tau=\qproet$, or
		\item $\tau=\proet$ if $S$ is locally Noetherian, or
		\item $\tau=\et$ if either $S$ is sousperfectoid and $f$ is smooth, or if $S$ is locally Noetherian.
	\end{enumerate}
	Then for any $n\in \N$, the following natural map
	on $S_\tau$ 
	is an almost isomorphism:
	\[ (R^nf_{\tau\ast}\mathbb L)\otimes_{\Z_p} \O^+_S/\varpi\to R^nf_{\tau\ast}(\mathbb L\otimes_{\Z_p} \O^+_X/\varpi).\]
\end{Theorem}
Here we use the notion of almost mathematics on $S_\tau$ explained in \S\ref{s:almost-mathematics}.

In comparison to \Cref{t:PCT-rigid}, this Theorem allows more general $S$ and more general topologies:
For example, for the v-topology, \Cref{t:PCT-perfectoid} is in effect a Primitive Comparison Theorem for the base-change $f':X'\to S'$ of $f$ to any perfectoid space $S'\to S$. This means that we leave the rigid world even if $f$ is a morphism of rigid spaces.

\begin{Remark}
	Regarding the finite type assumption in Theorem \ref{t:PCT-rigid} and \ref{t:PCT-perfectoid}, we note that properness only requires ``${}^+$weakly finite type'', whereas we impose the stronger assumption ``finite type'', see \cite[Definition 1.10.14]{huber2013etale}. This is necessary to ensure that the geometric fibres of $f$ are proper \textit{rigid} spaces. We thank 
	Yicheng Zhou for pointing this out to us.
	
	In fact, for $\tau=v$ or $\tau=\qproet$, we do not need to assume that $X$ and $S$ are represented by adic spaces: Instead, we can allow $f:X\to S$ to be any qcqs morphism of locally spatial diamonds such that for every point $s:\Spa(C,C^+)\to S$ valued in a non-archimedean field, the base-change $X\times_Ss\to \Spa(C,C^+)$ is a proper rigid space. This is useful to avoid problems with sheafiness in the non-Noetherian setup which arise for example when dealing with non-smooth proper   morphisms over perfectoid spaces.
\end{Remark}

For the proof of \Cref{t:PCT-perfectoid}, we use that by \cite[\S2.5]{huber2013etale}, there is a good notion of geometric points for analytic adic spaces. More generally, it is implicit in  \cite[\S11, \S14]{Sch18} that this works for locally spatial diamonds:

\begin{Definition}
	Let $X$ be a locally spatial diamond.
	For any $x_0\in |X|$, let $\mathcal N(X,x_0)$ be the category of pairs $(U,u)$ of a quasi-compact \'etale morphism $U\to X$ with a lift $u\in |U|$ of $x_0$. 
\end{Definition}
\begin{Lemma}
	There is an algebraically closed non-archimedean field $(C,C^+)$ with compatible maps $\Spa(C,C^+)\to U$ such that
	\[ \Spa(C,C^+)=\varprojlim_{U\in \mathcal N(X,x_0)}U\to X\]
	and this map sends the closed point of $\Spa(C,C^+)$ to $x_0$.
\end{Lemma}
\begin{Definition}
	A map $x:\Spa(C,C^+)\to X$ of this form is called a geometric point of $S_\et$. 
\end{Definition}
\begin{proof}
	By \cite[Lemma 11.22]{Sch18}, there exists a limit
	\[Y:=\varprojlim_{U\in \mathcal N(X,x_0)}U\to X\]
	in locally spatial diamonds
	with a distinguished point $y\in |Y|$ lifting $x_0$.
	By \cite[Propositions 11.23, 11.26]{Sch18}, $Y$ is strictly totally disconnected. On the other hand, any open neighbourhood of $y$ in $Y$ has to be the whole space, so $Y$ is also connected. By \cite[Proposition~7.16]{Sch18}, it follows that $Y=\Spa(C,C^+)$ where $(C,C^+)$ is an algebraically closed perfectoid field. 
\end{proof}
\begin{Definition}
	For any sheaf $F$ on $X_\et$, we define the stalk of $F$ at the geometric point $x$ \[ F_x:=\varinjlim_{U\in \mathcal N(X,x_0)} F(U).\]
	This is clearly functorial in $F$.
	Slightly more generally, for any morphism $s:\Spa(L,L^+)\to X$ where $(L,L^+)$ is an algebraically closed non-archimedean field, consider the category $\mathcal N(s)$ of quasi-compact \'etale maps $U\to X$ through which $s$ factors. In the above, this is equivalent to the category $\mathcal N(X,x_0)$ where $x_0$ is the image of the closed point of $\Spa(L,L^+)$ in $|X|$. Then we obtain a factorisation $\Spa(L,L^+)\to \Spa(C,C^+)\xrightarrow{x} X$. We then set $F_s:=F_x=\Gamma(s^{-1}F)$.
\end{Definition}
\begin{Lemma}[{\cite[Proposition 14.3]{Sch18}}]\label{l:check-isom-on-stalks}
	Let $X$ be a locally spatial diamond.
	A morphism $f:F\to G$ of abelian sheaves on $X_\et$ is an isomorphism if and only if the morphism of stalks $f_x:F_x\to G_x$ is an isomorphism at every geometric point $x$ of $X$.
\end{Lemma}

We can then prove the following diamantine version of \cite[Proposition~2.6.1]{Huber-generalization}:
\begin{Lemma}\label{l:stalk-of-et-sheaf-of-v-cohom}
	Let $S$ be an analytic adic space over $\Z_p$ and let $f:X\to S$ be a qcqs morphism of locally spatial diamonds. Let $\mathbb L$ be an $\F_p$-local system on $X_v$ and
	let $G=\O^+/\varpi\otimes \mathbb L$ or $G= \mathbb L$ on $X_v$. Let $F$ be the \'etale sheafification of the pre-sheaf on $S_\et$ sending $U\mapsto H^n_v(X\times_SU,G)$. Then for any geometric point $s:\Spa(C,C^+)\to S$ of $S_\et$, we have an isomorphism
	\[F_s=H^n_v(X\times_Ss,G).\]
	If $X\times_Ss\to \Spa(C,C^+)$ is a rigid space, this moreover equals $F_s=H^n_\et(X\times_Ss, G)$.
\end{Lemma}
\begin{proof}
	The statement is local on $S$, so we may assume that $S$ is affinoid. Then the assumptions guarantee that $X$ is spatial, so by \cite[Proposition 11.24]{Sch18} we can find a strictly totally disconnected  quasi-pro-\'etale perfectoid cover $\wt X\to X$. It is clear from the definition that any such map is automatically quasi-compact and separated, hence it follows from \cite[Lemma 7.19]{Sch18} that any space in the \v{C}ech nerve is still affinoid perfectoid.
	
	It is clear from the definition that we have
	\begin{equation}\label{eq:stalk-of-vsheaf-lemma}
		F_s=\clim_{U\in \mathcal N(s)} H^n_v(X\times_SU,G).
	\end{equation} 
	We can compute the right hand side using \v{C}ech cohomology of the cover $\wt X\to X$. As the same is true for $H^n_v(X\times_Ss,G)$, we may therefore reduce to the case that $X$ is perfectoid. Moreover, it suffices to consider maps $U\to S$ which are standard-\'etale, as these are cofinal in the system of all \'etale maps. Then the fibre products $X\times_SU$ are again affinoid perfectoid. Since we have $s=\plim U$, it is clear that in diamonds, we have
	\[ X\times_Ss=\plim_U X\times_SU.\]
	The case of $G=\mathbb L$ is now immediate from \cite[Proposition~14.9]{Sch18}. 
	
	For the case of $G=\O^+/\varpi\otimes \mathbb L$, we use the following definition and lemma:
	\begin{Definition}\label{d:approx-limit}
		Let $Z$ be an affinoid adic space over $K$ and let $(Z_i)_{i\in I}$ be a cofiltered inverse system of affinoid adic spaces over $K$ with compatible morphisms $Z\to Z_i$. We write
		\[ Z\approx \varprojlim_{i\in I} Z_i\]
		if the induced map $|Z|\to \varprojlim_i |Z_i|$ of the underlying topological spaces is a homeomorphism and the induced map $\varinjlim_i \O(Z_i)\to \O(Z)$ has dense image. 
	\end{Definition}
	\begin{Lemma}\label{l:approx-G}
		Assume that $Z\approx \varprojlim_{i\in I} Z_i$ are as in \Cref{d:approx-limit} and that $Z$ and all $Z_i$ are moreover affinoid perfectoid spaces. Let $\mathbb L$ be an $\F_p$-local system on $Z_{i,\et}$ for some $i\in I$. We also denote by $\mathbb L$  the pullbacks to $Z$ and $Z_j$ for $j\geq i$. Then for any $n\in \N$,
		\[H^n_\et(Z,\O^+/\varpi\otimes \mathbb L)=\varinjlim_{j\geq i}  H^n_\et(Z_j,\O^+/\varpi\otimes \mathbb L).\]
	\end{Lemma}
	\begin{proof}
		This is a special case of  \cite[Lemma~2.16]{heuer-G-torsors-perfectoid-spaces}. We summarise the argument for the reader's convenience: For any $i\in I$ and any étale morphism $U_i\to Z_i$ that is either finite étale or a rational open immersion, one easily verifies that  $Z\times_{Z_i}U_i\approx \varprojlim_{j\geq i} Z_j\times_{Z_i}U_i$ (see e.g.\ \cite[Lemma~3.13]{heuer-diamantine-Picard}). By a \cH-argument, we can therefore reduce to the case of $\mathbb L=\F_p$. For this we have \[H^0_\et(Z,\O^+)/\varpi=\varinjlim_{i\in I}  H^0_\et(Z_i,\O^+)/\varpi,\] this can easily be seen directly from the definition of $\approx$ (use that $f\in \O(Z)$ lies in $\O^+(Z)$ if and only if $|f(z)|\leq 1$ for each $z\in |Z|$). We refer to  \cite[Lemma~3.10]{heuer-diamantine-Picard} for details.
		
		Since $Z_{\etqcqs}=2\text{-}\varinjlim_{i} Z_{i,\etqcqs}$ by \cite[Proposition 6.4]{Sch18}, the case $n=0$ of the lemma follows by sheafification. The case of general $n$ follows by considering \cH-cohomology.
	\end{proof}
	
	As $X\times_SU$ is affinoid perfectoid, we have
	by \cite[Proposition 6.5]{Sch18}  that
	\[X\times_Ss\approx \plim_U X\times_SU.\]
	By \Cref{l:approx-G} and \Cref{p:c2.15}, this implies that
	\[H^n_v(X\times_Ss,\O^+/\varpi\otimes \mathbb L)=\clim_U H^n_v(X\times_SU,\O^+/\varpi\otimes \mathbb L).\]
	The right hand side equals $F_s$ by \eqref{eq:stalk-of-vsheaf-lemma}. This gives the desired isomorphism. Finally, 
	when $X\times_Ss$ is a rigid space,
	we have $H^n_\et(X\times_Ss, G)=H^n_v(X\times_Ss,G)$ by \Cref{p:c2.15}.
\end{proof}

\begin{proof}[Proof of \Cref{t:PCT-perfectoid}]
	Let $S'\to S$ be any morphism in $S_\tau$ where $S'$ is an adic space and form the base-change diagram
	\[
	\begin{tikzcd}
		X' \arrow[d, "g'"] \arrow[r, "f'"] & S' \arrow[d, "g"] \\
		X \arrow[r, "f"]                           & S.               
	\end{tikzcd}\]
	Let $F$ be the \'etale sheafification of the presheaf on $S'_\et$ defined by
	\begin{equation}\label{eq:sheaf-in-proof-relPCT}
	U\to H^n_\tau(X'\times_{S'}U,\O^+_{X'}/\varpi\otimes g'^\ast\mathbb L).
	\end{equation}
	We note that by \Cref{p:c2.15}, the given assumptions on $S$ in each case ensure that
	\[H^n_\tau(X'\times_{S'}U,\O^+_{X'}/\varpi\otimes g'^\ast\mathbb L)= H^n_v(X'\times_{S'}U,\O^+_{X'}/\varpi\otimes g'^\ast\mathbb L).\]
	We claim that the natural morphism of sheaves on $S'_\et$ \[\varphi:(R^nf'_{\et\ast}g'^{\ast}\mathbb L)\otimes \O^+_{S'}/\varpi\to F\]
	is an isomorphism. Due to \Cref{p:14.7-14.8}, 
	this implies the Theorem after a further sheafification on $S$ with respect to the $\tau$-topology. 
	
	To prove the claim, it suffices by \Cref{l:check-isom-on-stalks} to see that its stalk $\varphi_s$ at any geometric point $s:\Spa(C,C^+)\to S'$ is an isomorphism. By \Cref{l:stalk-of-et-sheaf-of-v-cohom}, this stalk identifies with
	\[ \varphi_s: H^n_\et(X'\times_{S'}s, \mathbb L)\otimes C^+/\varpi\to  H^n_\et(X'\times_{S'}s,\O^+_{X'}/\varpi\otimes \mathbb L).\]
	But $X'\times_{S'}s\to s$ is a rigid space over $\Spa(C,C^+)$. The map $\varphi_s$ is therefore an isomorphism by \Cref{t:PCT-rigid} when $\Char(K)=0$, and by \Cref{t:PC-in-char-p} when it is $\Char(K)=p$.
\end{proof}
\begin{Remark}
	If $X\to S$ is a morphism of rigid spaces,  one can alternatively prove the result by rigid approximation: Let $S'\to S$ be any morphism from an affinoid perfectoid space and let $S'\approx \varprojlim S_i$ be a rigid approximation over $S$. Then for any $n$, one can show that the map
	\[ \textstyle\clim_{i}H^n_{v}(X\times_SS_i,\O^+/\varpi)\to H^n_{v}(X\times_S{S'},\O^+/\varpi)\]
	is an isomorphism. Note that $X\times_SS_i$ is a rigid space. In particular, this shows that $H^n_{\et}(X\times_SS',\O^+/\varpi)=H^n_{v}(X\times_SS',\O^+/\varpi)$ (which is not a priori clear  -- recall that we consider $X\times_SS'$ as a diamond which need not come from a sheafy adic space). Consequently, with more work, the sheaf $F$ from \Cref{eq:sheaf-in-proof-relPCT} could in fact be identified with $R^nf_{\et\ast}( \O^+_{X'}/\varpi\otimes g'^\ast\mathbb L)$.
\end{Remark}

\section{Proper Base Change for locally constant sheaves}
As one of the many applications of  the Primitive Comparison Theorem, it is known that one can use it to prove Proper Base Change results for \'etale cohomology of adic spaces over $\Q_p$: This strategy can be found in the work of Bhatt--Hansen \cite[Theorem~3.15]{BhattHansen_Zar_constructible}, and it is also implicit in the work of Mann, who proves a very general Proper Base Change result for $\O^+/\varpi$-modules \cite[Theorem~1.2.4.(iv)]{mann2022padic}.

As an application of \Cref{t:PC-in-char-p},
we now explain an instance of this argument that works for analytic adic spaces over $\Z_p$,  giving new proper base change results in this generality.

Throughout this section, we let $(K,K^+)$ be any non-archimedean field over $\Z_p$ with pseudo-uniformiser $\varpi$.

\begin{Theorem}[Proper Base Change]\label{p:proper-base-change}
	Let $f:X\to S$ be a proper morphism of finite type of analytic adic spaces over $\Z_p$. Let $\mathbb L$ be an \'etale-locally constant torsion abelian sheaf on $X_\et$. Let $g:S'\to S$ be any morphism of adic spaces and form the base-change diagram in locally spatial diamonds:
	\[
	\begin{tikzcd}
		X' \arrow[d, "g'"] \arrow[r, "f'"] & S' \arrow[d, "g"] \\
		X \arrow[r, "f"]                           & S              
	\end{tikzcd}\]
	Then for any $n\in \N$, the following natural map is an isomorphism:
	\begin{equation}\label{eq:proper-base-change-map}
		g^{\ast}R^nf_{\et\ast}\mathbb L\to R^nf'_{\et\ast}g'^{\ast}\mathbb L.
	\end{equation}
\end{Theorem}
\begin{Remark}\label{r:Huber-proper-base-change}
	Before discussing the proof, let us detail the relation of \Cref{p:proper-base-change} to previous instances of non-archimedean Proper Base Change results in the literature:
	\begin{enumerate}
		\item The case  of \Cref{p:proper-base-change} of $\ell$-torsion coefficients for $\ell\neq p$ follows from work of Huber, as does the case when $g$ has dimension $0$ \cite[Theorem~4.4.1]{huber2013etale}. In general, the case of $p$-torsion coefficients was left open in \cite{huber2013etale}.
		\item Another closely related result, due to Scholze, is \cite[Corollary 16.10]{Sch18}, which gives the result when $f$ is any qcqs morphism of locally spatial diamonds and $\mathbb L$ is prime-to-$p$-torsion. It also gives the $p$-torsion case when $g$ is quasi-pro-\'etale.
		\item The case when $f$ and $g$ are morphisms of rigid spaces in characteristic $0$ is due to Bhatt--Hansen \cite[Theorem~3.15]{BhattHansen_Zar_constructible}, who more generally work with Zariski-constructible sheaves. On the other hand, we work with more general adic spaces: Indeed, over $\Q_p$, we are particularly interested in the case that $S'$ is perfectoid, which we use in \cite{HX} for applications to $p$-adic moduli spaces in non-abelian Hodge theory.
		\item In characteristic $0$, when $f$ is additionally smooth, \Cref{p:proper-base-change} also follows from a result of Mann \cite[Theorem~1.2.4.(iv)]{mann2022padic}. Mann's Theorem is a priori a base change for $\O^+/p$-modules. Under the additional assumptions, one can translate this to a Proper Base Change  for $\F_p$-modules, cf \cite[Theorem 3.9.23, Corollary 3.9.24]{mann2022padic}.
		\item In characteristic $p$, our proof of \Cref{p:proper-base-change} has some conceptual similarities with \mbox{Bhatt--Lurie}'s proof of Proper Base Change for $p$-torsion sheaves in the algebraic setting of schemes over $\F_p$, using their Riemann--Hilbert correspondence \cite[\S10.6]
		{BhattLurie}.
	\end{enumerate}
\end{Remark}
Like Bhatt--Hansen and Mann, we will use Primitive Comparison to deduce \Cref{p:proper-base-change} from a base-change statement for $\O^+/\varpi$-modules. In  our setup, the latter is given by the following observation:
Let $(C,C^+)$ be an algebraically closed complete extension of $(K,K^+)$.

\begin{Lemma}[{\cite[Proposition 4.2]{heuer-diamantine-Picard}}]\label{l:pct-stalks-O^+/p}
	Let $X$ be any proper rigid space over $(C,C^+)$. Let $\mathbb L$ be any $\F_p$-local system on $X$. Let $Y$ be any affinoid perfectoid space over $(C,C^+)$. Then
	\[H^n_\et(X\times Y,\mathbb L\otimes \O^+/\varpi)\aeq H^n_\et(X,\mathbb L\otimes \O^+/\varpi)\otimes \O^+/\varpi(Y).\]
\end{Lemma}
\begin{proof}
	This is a slight generalisation of a statement in the proof of \cite[Lemma 3.25]{BhattHansen_Zar_constructible}, as well as of \cite[Proposition 4.2]{heuer-diamantine-Picard}. Either proof goes through in this setting, as we now explain: Since $X$ is quasi-compact and separated, we can find a quasi-pro\'etale cover $\wt X\to X$ by an affinoid perfectoid space such that all spaces $\wt X_i$ in the \v{C}ech-nerve are still affinoid perfectoid. It follows that also $\wt X\times Y\to X\times Y$ is a cover with this property. We may assume that $\mathbb L$ becomes trivial on $\wt X$. Then by almost acyclicity, the cohomology of $\mathbb L\otimes \O^+/\varpi$ is almost acyclic on each $\wt X_i$ and $\wt X_i\times Y$. 
	Via \Cref{p:c2.15}, we can therefore compute:
	\begin{align*}
		H^n_\et(X\times Y,\mathbb L\otimes \O^+/\varpi)&=H^n_v(X\times Y,\mathbb L\otimes \O^+/\varpi)\aeq\check{H}^n(\wt X\times Y\to X\times Y,\mathbb L\otimes \O^+/\varpi)
		\\&\aeq\check{H}^n(\wt X\to X,\mathbb L\otimes \O^+/\varpi)\otimes \O^+/\varpi(Y)
		\\&\aeq H^n_\et(X,\mathbb L\otimes \O^+/\varpi)\otimes \O^+/\varpi(Y)
	\end{align*}
	where in the second line we have used that $\O^+(\wt X_i\times Y)/\varpi\aeq\O^+(\wt X_i)/\varpi\otimes \O^+(Y)/\varpi$ and that $\O^+(Y)/\varpi$ is flat over $\O_K/\varpi$.
\end{proof}
\begin{proof}[Proof of \Cref{p:proper-base-change}]
	The sheaf $\mathbb L$ is the direct sum of its $\ell$-power torsion subsheaves for all primes $\ell$. By \Cref{r:Huber-proper-base-change}.2, it remains to treat $p$-power torsion sheaves. 
	By considering $p$-torsion subsheaves, $\mathbb L$ is the successive extension of \'etale-locally free $\F_p$-modules, so we can reduce to this case.
	By \Cref{l:check-isom-on-stalks}, it suffices to prove that \eqref{eq:proper-base-change-map} is an isomorphism on stalks at geometric points $s'=\Spa(L,L^+)\to S'$ for algebraically closed non-archimedean fields $(L,L^+)$ over $(K,K^+)$. By \Cref{l:stalk-of-et-sheaf-of-v-cohom}, it thus suffices to prove that for $s:=g\circ s'$, the map
	\[ H^n_\et(X_s,\mathbb L)\to H^n_\et(X'_{s'},\mathbb L)\]
	is an almost isomorphism, where $X_s:=X\times_Ss$ and $X'_{s'}:=X'\times_{S'}s'$.  We have thus reduced to the case that $S=\Spa(C,C^+)$ and $S'=\Spa(L,L^+)$ are both geometric points.
	
	As an $\F_p$-vector space $W$ is trivial if and only if $W\otimes_{\F_p} L^+/\varpi$ is almost trivial, it suffices to see that the map
	\[ H^n_\et(X_s,\mathbb L)\otimes L^+/\varpi\to H^n_\et(X'_{s'},\mathbb L)\otimes L^+/\varpi\]
	obtained by tensoring with $L^+/\varpi$ is an almost isomorphism.
	Via the Primitive Comparison  \Cref{t:PCT-rigid}, this map gets identified in the almost category with the map
	\[H^n_\et(X_s,\mathbb L\otimes \O^+/\varpi)\otimes_{C^+/\varpi}L^+/\varpi\to H^n_\et(X'_{s'},\mathbb L\otimes \O^+/\varpi)\]
	This is an almost isomorphism by \Cref{l:pct-stalks-O^+/p}.
\end{proof}

We obtain the following analogue of \cite[Corollary~3.17(ii)]{Scholze_p-adicHodgeForRigid} for $S_\qproet$ and $S_v$:
\begin{Corollary}\label{t:PPBC-perfectoid}
	Let $f:X\to S$ be a proper morphism of finite type of analytic adic spaces over $\Z_p$.  Let $\mathbb L$ be an \'etale-locally constant torsion abelian sheaf on $X_\et$. Let $\tau\in \{\proet,\qproet,v\}$. Then for any $n\in \N$, the natural base-change map 
	\[\lambda^{\ast}R^nf_{\et\ast}\mathbb L\isomarrow R^nf_{\tau\ast}\lambda^{\ast}\mathbb L\]
	for the morphism $\lambda:S_\tau\to S_\et$ is an isomorphism
\end{Corollary}
Here, as before, when dealing with $\tau=\proet$, we always tacitly assume that $X$ is locally Noetherian, as this is assumed in the definition of $X_\proet$ in \cite[\S3]{Scholze_p-adicHodgeForRigid}.
\begin{proof}
	Let $g:S'\to S$ be a perfectoid space in $S_\tau$.
	By \Cref{p:14.7-14.8}, we have
	\[H^n_\tau(X\times_SS',\mathbb L)=H^n_\et(X\times_SS',\mathbb L).\]
	By \Cref{p:proper-base-change}, the \'etale sheafification of this over $S'_\et$ is  $g^{\ast}R^nf_{\et\ast}\mathbb L$. Upon $\tau$-sheafification, this shows the desired result.
\end{proof}
\begin{Theorem}[Künneth formula]\label{c:Kunneth}
	Let $f:X\to S$ and $g:Y\to S$ be morphisms of analytic adic spaces over $\Z_p$ such that $f$ is proper of finite type and $g$ is qcqs. Form the base-change diagram:
	\[\begin{tikzcd}
		X\times_SY \arrow[rd, "h"] \arrow[r, "f'"] \arrow[d, "g'"] & Y \arrow[d, "g"] \\
		X \arrow[r, "f"]                                           & S               
	\end{tikzcd}\]
	
	Let $n\in \N$ and let $\mathbb L$ be an  \'etale-locally constant $n$-torsion abelian sheaf on $X_\et$. Then for any $\Z/n$-module $E$ on $Y_\et$, we have a natural isomorphism
	\[ Rh_{\et\ast}(g'^\ast\mathbb L\otimes f'^\ast E)=Rf_{\et\ast}\mathbb L\otimes Rg_{\et\ast}E.\qedhere\]
\end{Theorem}
This extends the case of prime-to-$p$-torsion due to Berkovich  \cite[Corollary~7.7.3]{Berkovich_EtaleCohom}.
\begin{proof}
	The cup product defines a natural morphism from right to left. To see that this is a quasi-isomorphism, it suffices by \Cref{l:check-isom-on-stalks} to check this on stalks at any geometric point $x:\Spa(C,C^+)\to S$. By \Cref{l:stalk-of-et-sheaf-of-v-cohom}, this reduces us to the case that $S=\Spa(C,C^+)$.
	
	 With our preparations, we can now essentially follow the proof of \cite[Corollary~7.7.3]{Berkovich_EtaleCohom}: We first write 
	\begin{align*}
	 Rh_{\et\ast}(g'^\ast\mathbb L\otimes f'^\ast E)&= Rg_{\et\ast}Rf'_{\et\ast}(g'^\ast\mathbb L\otimes f'^\ast E).\\
	 \intertext{By \cite[
	Proposition 5.5.1]{huber2013etale}, we have $Rf'_{\et\ast}(g'^\ast\mathbb L\otimes f'^\ast E)=E\otimes Rf'_{\et\ast}g'^\ast\mathbb L$, so this is}
	&=Rg_{\et\ast}(E\otimes Rf'_{\et\ast}g'^\ast\mathbb L).\\
	\intertext{By \Cref{p:proper-base-change}, we have $Rf'_{\et\ast} g'^\ast\mathbb L=g^\ast Rf_{\et\ast}\mathbb L$, so this is}
	&=Rg_{\et\ast}(E\otimes g^\ast Rf_{\et\ast}\mathbb L).\\
	\intertext{We now use that by \Cref{t:finiteness}, the complex of $\Z/n$-modules $Rf_{\et\ast}\mathbb L$ is perfect. We can therefore invoke the projection  formula \cite[0944]{StacksProject} to see that the above term is}
	&=Rf_{\et\ast}\mathbb L\otimes Rg_{\et\ast}E. \qedhere
	\end{align*}
\end{proof}
\begin{Corollary}
	Let $C$ be an algebraically closed non-archimedean field.
	Let $X$ be a proper rigid space over $C$ and let $Y$ any qcqs analytic adic spaces over $C$. Then for any $n\in \N$,
	\[ R\Gamma_\et(X\times Y,\Z/n\Z)= R\Gamma_\et(X,\Z/n\Z)\otimes R\Gamma_\et(Y,\Z/n\Z).\]
\end{Corollary}

Finally, we note that in algebraic setups, \Cref{p:proper-base-change} generalises to constructible sheaves. For simplicity, let us restrict attention to the case where $S'$ is stably uniform.

\begin{Proposition}\label{t:PBC-perfectoid}
	Let $f:X\to S$ be an algebraisable proper morphism of rigid spaces over $(K,K^+)$. Let $\mathbb L$ be any algebraic Zariski-constructible torsion abelian sheaf on $X_\et$. Let $g:S'\to S$ be a morphism of adic spaces where $S'$ is stably uniform and form the base-change diagram in locally spatial diamonds:
		\[
	\begin{tikzcd}
		X' \arrow[d, "g'"] \arrow[r, "f'"] & S' \arrow[d, "g"] \\
		X \arrow[r, "f"]                           & S              
	\end{tikzcd}\]
	Then for any $n\in \N$, the following natural base change maps
	are isomorphisms:
	\begin{align*}
		1)&&g^{\ast}R^nf_{\et\ast}\mathbb L&\isomarrow R^nf'_{\et\ast}g'^{\ast}\mathbb L,&&\\
		2)&&\lambda^\ast R^nf_{\et\ast}\mathbb L&\isomarrow R^nf_{v\ast}\lambda^\ast \mathbb L&&
	\end{align*}
	where $\lambda:X_v\to X_\et$ is the natural morphism.
\end{Proposition}
\begin{proof}
	The statement is local on $S'$, so we may assume that $S'$ and $S$ are affinoid. We can moreover assume that $S$ is reduced as passing to the underlying reduced space does not change either \'etale or v-sites. Recall that any reduced rigid space is stably uniform. By \cite[Proposition 3.17]{heuer-diamantine-Picard}, we can therefore find a  ``rigid approximation'', namely a cofiltered inverse system $(S_i\to S)_{i\in I}$ of smooth morphisms from affinoid rigid spaces such that 
	\[S'\approx \varprojlim_{i\in I} S_i.\] 
	In particular, we then have $X'=\varprojlim_{i\in I} X\times_SS_i$ in locally spatial diamonds.
	This reduces us to the case where $S'$ is rigid and there is $m$ such that $g$ factors into an \'etale map $S'\to S\times \mathbb A^m$ composed with the projection $\pi:S\times \mathbb A^m\to S$. We can therefore reduce to the case that $S'=S\times \mathbb A^m$ and $g:S'\to S$ is algebraic. Let $f^{\alg}:X^{\alg}\to S^{\alg}$ be the morphism of schemes for which $f^{\alg,\an}=f$, and similarly for $\pi$. By \cite[Theorem~3.7.2]{huber2013etale} and the preceding discussion, we have
	\[\pi^{\ast}R^nf_{\et\ast}\mathbb L=\pi^{\ast}(R^nf^\mathrm{alg}_{\et\ast}\mathbb L^\mathrm{alg})^{\an}=(\pi^{\mathrm{alg}\ast}R^nf^\mathrm{alg}_{\et\ast}\mathbb L^\mathrm{alg})^{\an},\]
	\[R^nf'_{\et\ast}\pi'^\ast\mathbb L= (R^nf'^{\mathrm{alg}}_{\et\ast}\pi'^{\mathrm{alg}\ast}\mathbb L^\mathrm{alg})^{\an}.\]
	Part 1) now follows from classical Proper Base Change. 
	
	Part 2) follows from 1) by letting $S'$ range through all perfectoid spaces over $S$.
\end{proof}


\begin{thebibliography}{dJvdP96}
	
	\bibitem[AIP18]{AIP}
	F.~Andreatta, A.~Iovita,  V.~Pilloni.
	\newblock {L}e halo spectral.
	\newblock {\em Ann. Sci. \'{E}c. Norm. Sup\'{e}r. (4)}, 51(3):603--655, 2018.
	
	\bibitem[Ber93]{Berkovich_EtaleCohom}
	V.~G. Berkovich.
	\newblock \'{E}tale cohomology for non-{A}rchimedean analytic spaces.
	\newblock {\em Inst. Hautes \'{E}tudes Sci. Publ. Math.}, (78):5--161 (1994),
	1993.
	
	\bibitem[BH22]{BhattHansen_Zar_constructible}
	B.~Bhatt,  D.~Hansen.
	\newblock The six functors for {Z}ariski-constructible sheaves in rigid
	geometry.
	\newblock {\em Compos. Math.}, 158(2):437--482, 2022.
	
	\bibitem[BL19]{BhattLurie}
	B.~Bhatt,  J.~Lurie.
	\newblock A {R}iemann-{H}ilbert correspondence in positive characteristic.
	\newblock {\em Camb. J. Math.}, 7(1-2):71--217, 2019.
	
	\bibitem[BS15]{bhatt-scholze-proetale}
	B.~Bhatt,  P.~Scholze.
	\newblock {T}he pro-\'{e}tale topology for schemes.
	\newblock {\em Ast\'{e}risque}, (369):99--201, 2015.
	
	\bibitem[dJvdP96]{deJongvdPut}
	A.~J. de~Jong,  M.~van~der Put.
	\newblock \'{E}tale cohomology of rigid analytic spaces.
	\newblock {\em Doc. Math.}, 1:No. 01, 1--56, 1996.
	
	\bibitem[Fal02]{Faltingsalmostetale}
	G.~Faltings.
	\newblock Almost \'{e}tale extensions.
	\newblock {\em Ast\'{e}risque}, (279):185--270, 2002.
	\newblock Cohomologies $p$-adiques et applications arithm\'{e}tiques, II.
	
	\bibitem[GR03]{GabberRamero}
	O.~Gabber,  L.~Ramero.
	\newblock {\em {A}lmost ring theory}, volume 1800 of {\em Lecture Notes in
		Mathematics}.
	\newblock Springer-Verlag, Berlin, 2003.
	
	\bibitem[Heu21]{heuer-diamantine-Picard}
	B.~Heuer.
	\newblock {D}iamantine {P}icard functors of rigid spaces.
	\newblock {\em Preprint, arXiv:2103.16557}, 2021.
	
	\bibitem[Heu22]{heuer-G-torsors-perfectoid-spaces}
	B.~Heuer.
	\newblock {$G$}-torsors on perfectoid spaces.
	\newblock {\em Preprint, arXiv:2207.07623}, 2022.
	
	\bibitem[HK]{HK}
	D.~Hansen,  K.~S. Kedlaya.
	\newblock Sheafiness criteria for {H}uber rings.
	\newblock {\em Preprint, available
		\href{https://kskedlaya.org/papers/criteria.pdf}{at this link}}.
	
	\bibitem[Hub94]{Huber-generalization}
	R.~Huber.
	\newblock A generalization of formal schemes and rigid analytic varieties.
	\newblock {\em Math. Z.}, 217(4):513--551, 1994.
	
	\bibitem[Hub96]{huber2013etale}
	R.~Huber.
	\newblock {\em \'{E}tale cohomology of rigid analytic varieties and adic
		spaces}.
	\newblock Aspects of Mathematics, E30. Friedr. Vieweg \& Sohn, Braunschweig,
	1996.
	
	\bibitem[Hub07]{HuberFiniteness}
	R.~Huber.
	\newblock A finiteness result for the compactly supported cohomology of rigid
	analytic varieties. {II}.
	\newblock {\em Ann. Inst. Fourier (Grenoble)}, 57(3):973--1017, 2007.
	
	\bibitem[HX24]{HX}
	B.~Heuer,  D.~Xu.
	\newblock $p$-adic non-abelian {H}odge theory for curves via moduli stacks.
	\newblock {\em Preprint, arXiv.org:2402.01365}, 2024.
	
	\bibitem[Kat73]{p-adicMSMF}
	N.~M. Katz.
	\newblock {$p$}-adic properties of modular schemes and modular forms.
	\newblock In {\em Modular functions of one variable, {III}}, pages 69--190.
	Lecture Notes in Math., Vol. 350. Springer, Berlin, 1973.
	
	\bibitem[Ked10]{Kedlaya_differential-equations}
	K.~S. Kedlaya.
	\newblock {\em {$p$}-adic differential equations}, volume 125 of {\em Cambridge
		Studies in Advanced Mathematics}.
	\newblock Cambridge University Press, Cambridge, 2010.
	
	\bibitem[Man22]{mann2022padic}
	L.~Mann.
	\newblock A $p$-adic 6-functor formalism in rigid-analytic geometry.
	\newblock {\em Preprint, arXiv:2206.02022}, 2022.
	
	\bibitem[Mum70]{MumfordAV}
	D.~Mumford.
	\newblock {\em {A}belian varieties}.
	\newblock Tata Institute of Fundamental Research Studies in Mathematics, No. 5.
	Published for the Tata Institute of Fundamental Research, Bombay; Oxford
	University Press, London, 1970.
	
	\bibitem[Pan22]{PanLocallyAnalytic}
	L.~Pan.
	\newblock On locally analytic vectors of the completed cohomology of modular
	curves.
	\newblock {\em Forum Math. Pi}, 10:Paper No. e7, 82, 2022.
	
	\bibitem[Sch12]{perfectoid-spaces}
	P.~Scholze.
	\newblock {P}erfectoid spaces.
	\newblock {\em Publ. Math. Inst. Hautes \'{E}tudes Sci.}, 116:245--313, 2012.
	
	\bibitem[Sch13a]{Scholze_p-adicHodgeForRigid}
	P.~Scholze.
	\newblock {$p$}-adic {H}odge theory for rigid-analytic varieties.
	\newblock {\em Forum Math. Pi}, 1:e1, 77, 2013.
	
	\bibitem[Sch13b]{ScholzeSurvey}
	P.~Scholze.
	\newblock {P}erfectoid spaces: {A} survey.
	\newblock In {\em Current developments in mathematics 2012}, pages 193--227.
	Int. Press, Somerville, MA, 2013.
	
	\bibitem[Sch15]{torsion}
	P.~Scholze.
	\newblock {O}n torsion in the cohomology of locally symmetric varieties.
	\newblock {\em Ann. of Math. (2)}, 182(3):945--1066, 2015.
	
	\bibitem[Sch22]{Sch18}
	P.~Scholze.
	\newblock {\'E}tale cohomology of diamonds.
	\newblock {\em Preprint, arXiv:1709.07343}, 2022.
	
	\bibitem[Sta]{StacksProject}
	{The Stacks Project Authors}.
	\newblock {T}he stacks project.
	\newblock 2023.
	
	\bibitem[SW20]{ScholzeBerkeleyLectureNotes}
	P.~Scholze,  J.~Weinstein.
	\newblock {\em {B}erkeley {L}ectures on $p$-adic {G}eometry}.
	\newblock Annals of Mathematics Studies. Princeton University Press, Princeton,
	NJ, 2020.
	
	\bibitem[Zav21]{Zavyalov_acoh}
	B.~Zavyalov.
	\newblock Almost coherent modules and almost coherent sheaves.
	\newblock {\em Preprint, arXiv:2110.10773}, 2021.
	
\end{thebibliography}
\end{document}